\documentclass[10pt,a4paper,twoside,reqno]{amsart}

\usepackage{amsmath}
\usepackage{amssymb}
\usepackage{amsfonts}
\usepackage{amsthm}
\usepackage{hyperref}
\usepackage{mathrsfs}
\usepackage{dsfont}
\usepackage{mathtools}
\usepackage{graphicx}
\usepackage{bm}
\usepackage{todonotes}
\usepackage{esint}
\newcommand*{\mailto}[1]{\href{mailto:#1}{\nolinkurl{#1}}}

\usepackage{color}

\definecolor{darkgreen}{rgb}{0.5,0.25,0}
\definecolor{darkblue}{rgb}{0,0,1}
\definecolor{answerblue}{rgb}{0,0,0.75}

\hypersetup{colorlinks,breaklinks,
            linkcolor=darkblue,urlcolor=darkblue,
            anchorcolor=darkblue,citecolor=darkblue}

\newcommand{\ep}{\varepsilon}

\renewcommand{\ge}{\geqslant}
\renewcommand{\le}{\leqslant}

\renewcommand{\d}{\mathrm{d}}

\newcommand{\R}{\mathbb{R}}

\newcommand{\abs}[1]{\left | #1 \right |}
\newcommand{\norm}[1]{\left \| #1 \right \|}
\newcommand{\bk}[1]{ \left(  #1 \right)}

\theoremstyle{theorem}
\newtheorem{thm}{Theorem}[section]
\newtheorem{prop}[thm]{Proposition}
\newtheorem{lem}[thm]{Lemma}
\newtheorem{cor}[thm]{Corollary}

\theoremstyle{definition}
\newtheorem{defin}{Definition}[section]
\theoremstyle{remark}
\newtheorem{rem}{Remark}[section]

\allowdisplaybreaks

\title[sewing and splitting-up for RDEs]
{Convergence rate in the splitting-up method for rough differential equations}

\author[Pang]{Peter H.C. Pang}
\address[P.H.C. Pang]{Department of Mathematics\\
   University of Oslo\\
  NO-0316 Oslo\\ Norway}
\email{\mailto{ptr@math.uio.no}}


\date{\today}
\thanks{The author is grateful for the support of 
	the Norwegian Research Council via 
	the project {\em INICE} (301538).}

\begin{document}
\begin{abstract} 
In this note we construct solutions to 
rough differential equations $\d Y = f(Y) \,\d X$ 
with a driver $X \in C^\alpha([0,T];\R^d)$, 
$\frac13 < \alpha \le \frac12$, using a 
splitting-up scheme. We show convergence of 
our scheme to solutions in the sense of Davie 
by a new argument and give a rate of convergence.
\end{abstract}
\maketitle

\section{Introduction}
\subsection{Rough differential equations}

Suppose $X \in C^\alpha([0,T];\R^d)$ for a 
fixed $\alpha \in (1/3,1/2]$ and consider 
the rough differential equation 
\begin{align}\label{eq:rde1}
\d Y = f(Y) \,\d X, 
\end{align}
where $f$ is a fixed function in $C^{\gamma -1}_b(\R^d)$, 
and $\gamma > \frac1\alpha \ge 2$. 
Standard rough paths theory \cite[Ch. 4, 8]{FH2020} informs us 
that in order to make sense of this equation, 
we must include additional information in 
the form of the area 
$$
 \int_s^t X_{s,r}\,\d X_r := \mathbb{X}_{s,t}, 
$$
satisfying $\sup_{s \not= t} |\mathbb{X}_{s,t}|/|t - s|^{2 \alpha} < \infty$ 
and Chen's relation. 
In this way, the augmented problem 
$$
\d Y = f(Y) \,\d {\bf X}, \qquad {\bf X} := (X,\mathbb{X})
$$
can be understood via a sewing lemma 
to be asking for $Y$ such that 
$$
Y_{s,t} = \int_s^t f(Y_r)\,\d {\bf X}_r := 
\lim_{\norm{P} \downarrow 0} \sum_{t_i \in P } f(Y_{t_i}) X_{t_i,t_{i + 1}} 
	+  f(Y_{t_i}) f'(Y_{t_i}) \mathbb{X}_{t_i,t_{i + 1}},
$$
for every $s,t \in [0,T]$, and $P$ are finite partitions 
of $[s,t]$ with mesh size $\le \norm{P}$. 
The problem being thus understood, 
it can be solved via Picard iterations. In this note, we 
revisit the problem of existence of solutions to the 
RDE \eqref{eq:rde1} by considering a computational way 
of coming to assign a meaning and a solution to 
\eqref{eq:rde1}, similar to the classical paper 
\cite{Dav2008} (see also \cite[Section 4]{GL1997}).

\subsection{Operator splitting}

One fruitful way 
of constructing numerical schemes for 
(partial) differential equations is known 
as the {\em operator splitting}, or {\em fractional steps} method. 
A careful introduction to this in the context 
of deterministic PDEs can be found in \cite{HKLR2010}.
Consider a problem of the form 
$$
\frac{\d Y}{\d t} = \mathcal{A}(Y),
$$
where $\mathcal{A}$ is an operator 
that can be deomposed into a sum of 
simpler operators $\mathcal{A} 
	= \mathcal{A}_1 + \mathcal{A}_2$, 
(often) exhibiting qualitatively different behaviours. 
Let $\mathcal{S}_t$ be the corresponding 
solution operator so that a solution can be 
written as $Y(t) = \mathcal{S}_t \overline{Y}$ 
for an initial condition $\overline{Y}$. 
 For $j = 1,2$, let $\mathcal{S}_t^j$ be the solution 
operator for the equation
\begin{align*}
\frac{\d Y^j}{\d t} = \mathcal{A}_j(Y^j).
\end{align*}
It stands to reason that the solution $Y$ can 
be approximated by propagating the initial 
condition $\overline{Y}$ repeatedly using 
$\mathcal{S}_{{h}}^2\mathcal{S}_{{h}}^1$ 
for small time steps ${h}$, culminating in a 
Trotter formula of the form:
\begin{align*}
\mathcal{S}_t = \lim_{n \to \infty}
	\Big(\mathcal{S}_{t/n}^2 \mathcal{S}_{t/n}^1\Big)^n.
\end{align*}
Similar to a sewing lemma, such a formula 
tells us that on small time steps, 
the integral $\mathcal{S}_{{h}}\overline{Y}$ of the full 
problem can be well approximated by the 
composed integral $\mathcal{S}_{{h}}^2\mathcal{S}_{{h}}^1\overline{Y}$.

Let $N = [T/h]$ be a large integer. 
For $0 \le j \le N$, set $t_j := j {h}$. 
Turning again to the differential equation \eqref{eq:rde1}, 
we can consider the following approximating system
which contains the second order information $Z$ as 
an interlacing step:
\begin{equation}\label{eq:apsystem}
\begin{aligned}
Y^{h,1}_t &= Y^{h,2}_{t_j^-}
	+ f(Y^{h,1}_{t_j^-}) X_{t_j,t}, \quad t \in (t_j, t_{j + 1}],\\
Y^{h,2}_t &= Y^{h,1}_{t_{j + 1}}
	+ Z(Y^{h,1}_{t_{j + 1}})_{t_j ,t}, \quad t \in [t_j, t_{j + 1}),
\end{aligned}
\end{equation}
and appended with the initial condition $Y^{h,2}_0 = \overline{Y} \in \R^d$. 
(Since the standard notation in rough paths 
theory is to put two temporal subscripts to 
indicate the difference $X_{t,s} = X(t) - X(s)$, 
as well as for more general functions taking 
in two temporal arguments, we emphasise that 
$Z$ takes in three arguments, and $Z(Y_{t})_{r,s}$ 
does not mean a difference $Z(Y_t)_s - Z(Y_t)_r$.)
For a functon $g$ continuous on $(t_j, t_{j + 1})$, 
let $g_{t_j^+}$ denote the limit to $t_j$ from above. 
From the continuity of $X$, we find that 
\begin{align*}
Y^{h,1}_{t_j^+} &= Y^{h,2}_{t_j^-}, \qquad \text{and similarly, }\qquad
Y^{h,2}_{t_j} = Y^{h,1}_{t_{j + 1}}, 
\end{align*}
assuming $Z(\cdot)_{t,t} \equiv 0$.

Associated with $Y^{h,i}$, $i =1,2$, is an 
auxiliary solution $Y^h$ that runs at twice the rate, 
and has the advantage of being continuous at 
the grid points $t_j$ (at least, if $Z$ is continuous 
in its third argument). For $t \in (t_j,t_{j + 1}]$, we define:
\begin{align}\label{eq:joinedup}
Y^{h}_t := 
\begin{cases}
	Y^{h,1}_{t_j + 2(t - t_j)} & t \in (t_j, t_{j + 1/2}]\\
	Y^{h,2}_{t_j + 2(t - t_{j + 1/2})} & t \in (t_{j + 1/2}, t_{j + 1}]
\end{cases},
\end{align}
where $t_{j + 1/2} := t_j + \frac12 {h}$. 
We have $Y^h_{t_j} = Y^{h,1}_{t_j^+} = Y^{h,2}_{t_j^-}$. 
From the first equation of \eqref{eq:apsystem}, 
$Y^{h,1}$ is as smooth as $X$ 
in the interior of each interval of the partition. 
We can therefore hope that $Y^{h}$converges to a 
function in $C^\alpha([0,T];\R^d)$ that can be 
understood as a solution to \eqref{eq:rde1}. 
Following \cite{Dav2008}, we define solutions as follows:
\begin{defin}\label{def:solution}
A function $Y \in C^\alpha([0,T];\R^d)$ is a solution to 
the RDE \eqref{eq:rde1} if  for any $0 \le s \le t \le T$,
\begin{align*}
Y_{s,t} -  f(Y_s) X_{s,t} - Z(Y_s)_{s,t} = o(\abs{t - s}).
\end{align*}
\end{defin}

We delineate algebraic and analytic conditions on $Z$ 
under which $Y^{h}$ indeed converges 
and carry out the convergence argument in 
Section \ref{sec:wellposed}. 

The application of the splitting method to stochastic 
equations $\d u = F \,\d t + G \,\d W$ has also been widely 
studied (see, e.g., \cite{FO2011,GK2003,KS2018}). 
An intuitive way to split up the problem is to 
consider $\d u = F\,\d t $ and $\d u = G\,\d W$ on separate 
time-steps.

\subsection{Numerical analysis of RDEs}

Computational methods, especially explicit 
methods, of solving \eqref{eq:rde1} 
are very natural in light of the sewing lemma, and 
have been studied by various authors, including recently,  
\cite{BBL2023,RR2022}. The associated literature 
on computational methods for SDEs is vast. 
In the standard reference \cite[Chapter 10]{FV2016}, 
a thorough study was made of local higher order 
Euler methods. Our construction is reminescent of the 
classical work \cite{Dav2008} on which \cite[Chapter 10]{FV2016}
is also based. All of our calculations can be extended without 
change to the case where $Y$ takes values in 
$\R^n$, $X$ takes values in $\R^d$, and $f$ 
takes values in $\R^{n \times d}$.

In \cite{Dav2008}, Davie constructed solutions to 
\eqref{eq:rde1} using a Milstein scheme, or a 
second-order Euler scheme, using the area $\mathbb{X}$. 
Following \cite{Lyo1998}, driving paths 
$X$ of $p$-variation were considered in \cite{Dav2008}, 
with $2 \le p < 3$. Davie showed the more refined 
result that there is 
uniqueness only holds generally for $f \in C^p$. 
Two steps of a splitting scheme for ODEs 
added together is a slightly shifted Milstein scheme, 
interpolated between grid-points. 
We offer a different argument of convergence 
via the splitting scheme, which also provides a rate of convergence.

As earlier alluded to, our splitting scheme, 
like Davie's Milstein scheme, is explicit. As such it 
cannot witness non-uniqueness when $\gamma < 
	\frac1\alpha + 1$. 
It may of interest in subsequent work to analyse the 
convergence of a splitting scheme that is implicit in 
its first equation in subsequent work.

\section{Convergence and existence}\label{sec:wellposed}

\subsection{Assumptions on $Z$}
As indicated, the primary object of this note 
is to use assumptions on $Z$ to 
close estimates and give us convergence 
of approximating solutions to the operator 
splitting scheme in $C^\alpha([0,T];\R^d)$ 
as $h \downarrow 0$. We house our 
assumptions here and shall show their sufficiency 
in forthcoming calculations. 

\begin{enumerate}

\item 
For every $0 \le s \le t \le T$, uniformly in $x \in \R^d$, 
\begin{align}\label{eq:Z_st}
\abs{Z(x)_{s,t}} \lesssim \abs{t - s}^{2 \alpha}.
\end{align}

\item 
For every $0 \le s \le t \le T$, $x,y \in \R^d$, 
\begin{align}\label{eq:Z_xy}
\abs{Z(x)_{s,t} - Z(y)_{s,t}} \lesssim \abs{x - y}^{\gamma - 2} \abs{t - s}^{2 \alpha}.
\end{align}

\item 
For every $0 \le s \le u\le t \le T$, uniformly in $x \in \R^d$, 
define $\delta Z(x)_{s,u,t} := Z(x)_{s,t} - Z(x)_{s,u} - Z(x)_{u,t}$ as usual; 
we require
\begin{equation}\label{eq:Z_sut}
\begin{aligned}
&\Big|\delta Z(x)_{s,u,t} 
	- \big( \nabla f(x)  f(x) X_{s,u}\otimes X_{u,t}
	+ \nabla f(x) Z(x)_{s,u} X_{u,t}
	\big)\Big|\lesssim \abs{t - s}^{3\alpha}.
\end{aligned}
\end{equation}
\end{enumerate}

Traditionally, $Z(x)_{s,t} = f(x)  \nabla f(x) \mathbb{X}_{s,t} $.
This is now replaced by the analytic conditions 
\eqref{eq:Z_st} and \eqref{eq:Z_xy}.
The algebraic condition  \eqref{eq:Z_sut} matches 
the expected coboundary operation 
$\delta  Z_{s,u,t}$, 
with 
\begin{align*}
\delta  Z_{s,u,t}
	= f(x) \nabla f(x) X_{s,u} \otimes X_{u,t},
\end{align*}
being given by the first subinterval $[s,u]$ 
on the time step $[s,t]$ on the discrete level \cite{Gub2004}. 
Given the bound \eqref{eq:Z_st}, the final term 
$\nabla f(x) Z(x)_{s,t} X_{u,t}$ in the parentheses 
on the left of 
\eqref{eq:Z_sut} is aesthetic. 
Our calculations 
shall reveal that in fact, we can slightly relax $2\alpha$ in 
the exponent of $\abs{t - s}$ on 
the right of \eqref{eq:Z_st} 
(and the $3\alpha$ of \eqref{eq:Z_sut}) above to  
$(\gamma - 1) \alpha$ (resp., $\gamma \alpha$),  
and still arrive at the conclusions of Lemma 
\ref{thm:Yhh2_comparison},  Corollary 
\ref{thm:cauchy}, and Section \ref{sec:sew} below. 

\subsection{Comparision of $Y^h$ and $Y^{h/2}$ 
	at the same time point}\label{sec:sametime}
We first compare $Y^h$ and $Y^{h/2}$ in 
$C_b([0,t];\R^d)$, seeking an estimate in terms of $h$. 
We now show that:
\begin{prop}\label{thm:Yhh2_comparison}
Let $f \in C^{\gamma - 1}_b(\R)$ with $3 \ge \gamma > \frac1\alpha$. 
Let $Y^h$ and $Y^{h/2}$ be recursively defined via 
\eqref{eq:apsystem} and \eqref{eq:joinedup}. 
For a sufficiently small $h$, we have the estimate:
\begin{align}\label{eq:unif_estm_r}
\sup_{t \in [0,T]} \big|Y^h_t - Y^{h/2}_t\big| = O(h^{\gamma \alpha - 1}).
\end{align}
\end{prop}

\begin{proof}
We will treat only the case $\gamma \le 3$, 
as the alternative can be treated as if $\gamma  = 3$. 
From \eqref{eq:apsystem}, 
\begin{align*}
Y^h_{t_{j + 1}} = Y^h_{t_j} + f(Y^h_{t_j}) X_{t_j,t_{j + 1}}
	+ Z(Y^h_{t_{j + 1/2}})_{t_j,t_{j + 1}}.
\end{align*}
Since $t_{j + 1} - t_j =  h = 2r$, $Y^{h/2}$ runs through 
two time steps of size $r$ over $[t_j,t_{j + 1}]$:
\begin{align*}
Y^{h/2}_{t_{j + 1}} &= Y^{h/2}_{t_j} + f(Y^{h/2}_{t_j}) X_{t_j,t_{j + 1/2}}
	+ Z(Y^{h/2}_{t_{j + 1/4}})_{t_j,t_{j + 1/2}}\\
	&\qquad \qquad\,\,+ f(Y^{h/2}_{t_{j + 1/2}}) X_{t_{j + 1/2},t_{j + 1}} 
	+ Z(Y^{h/2}_{t_{j + 3/4}})_{t_{j + 1/2}, t_{j + 1}}.
\end{align*}

Taking a difference gives us:
\begin{align}\label{eq:localerr}
&Y^h_{t_{j + 1}} - Y^{h/2}_{t_{j + 1}}
=  Y^h_{t_j} - Y^{h/2}_{t_j} + I_1 + I_2, 
\end{align}
where
\begin{align*}
I_1& :=  f(Y^h_{t_j}) X_{t_j,t_{j + 1}}
	- f(Y^{h/2}_{t_j})X_{t_j,t_{j + 1/2}}
	-  f(Y^{h/2}_{t_{j + 1/2}}) X_{t_{j + 1/2},t_{j + 1}}, \\
I_2& :=  Z(Y^h_{t_{j + 1/2}})_{t_j,t_{j + 1}} 
	- Z(Y^{h/2}_{t_{j + 1/4}})_{t_j,t_{j + 1/2}}
	-  Z(Y^{h/2}_{t_{j + 3/4}})_{t_{j + 1/2}, t_{j + 1}}.
\end{align*}

We write $I_1$ as the sum $A_1 + A_2 + A_3$, where
\begin{align}\label{eq:I1_dec}
\begin{aligned}
A_1& := \Big( f(Y^h_{t_j}) - f(Y^{h/2}_{t_j})\Big)  X_{t_j,t_{j + 1}},\\
A_2 &:= -\Big[ f(Y^{h/2})_{t_j,t_{j + 1/2}} 
	 - \nabla f(Y^{h/2}_{t_j}) \cdot
	Y^{h/2}_{t_j,t_{j + 1/2}} \Big]
		 X_{t_{j + 1/2},t_{j + 1}},\\
A_3 & := - \nabla f(Y^{h/2}_{t_j})\cdot 
 	Y^{h/2}_{t_j,t_{j + 1/2}} 
	 X_{t_{j + 1/2},t_{j + 1}}.
\end{aligned}
\end{align}

We leave $A_1$ be.
If  $\gamma \ge 3$, the term $A_2$ can be written as
\begin{align*}
A_2 &= - \frac12 \nabla^2 f(\xi) :
	Y^{h/2}_{t_j,t_{j + 1/2}} \otimes Y^{h/2}_{t_j,t_{j + 1/2}} 
		X_{t_{j + 1/2},t_{j + 1}},
\end{align*}
using Taylor's theorem, for a $\xi = Y^{h/2}_{t_j} +  \theta 
	\bk{Y^{h/2}_{t_{j + 1/2}} - Y^{h/2}_{t_j}}$ 
for some $\theta \in [0,1]$. By \eqref{eq:apsystem}, 
each of the terms $Y^{h/2}_{t_j,t_{j + 1/2}}$ 
can be expanded as
\begin{align*}
 f(Y^{h/2}_{t_j}) X_{t_j,t_{j + 1/2}} 
 	+  Z(Y^{h/2}_{t_{j + 1/4}})_{t_j,t_{j + 1/2}}.
\end{align*}
Using the assumptions 
$X \in C^\alpha$ and \eqref{eq:Z_st}, 
$A_2 \le \frac12 \norm{\nabla^2 f}_{L^\infty} 
	\bk{\norm{f}_{L^\infty} + C} h^{3\alpha}$.
In any case ($\gamma < 3$), 
\begin{align}\label{eq:A2_estm}
A_2 \le \frac12 \abs{f}_{C^{\gamma - 1}} 
	\bk{\norm{f}_{L^\infty} + C} h^{\gamma \alpha}.
\end{align}

Using \eqref{eq:Z_st} and $X \in C^\alpha$, 
for any $g \in C^1_b(\R^d)$, $|\hat{h}| \le h$, we always have
\begin{align*}
\big|g(Y^{h/2})_{t_j,t_j + \hat{h}} \big| \lesssim h^\alpha.
\end{align*}
which also implies via \eqref{eq:Z_xy} that 
\begin{align}\label{eq:obs2}
\big|Z(Y^{h/2}_{t_{j} + \hat{h}})_{s,t} - Z(Y^{h/2}_{t_j})_{s,t}\big|
\lesssim\abs{t - s}^{2\alpha} h^{(\gamma - 2)\alpha}.
\end{align}

For $A_3$, we get:
\begin{align*}
A_3 & = -f(Y^{h/2}_{t_j})\nabla f(Y^{h/2}_{t_j}) 
	X_{t_j,t_{j + 1/2}}\otimes
	X_{t_{j + 1/2},t_{j + 1}} \\
	& \quad\,\, - \nabla f(Y^{h/2}_{t_j}) 
	Z(Y^{h/2}_{t_j})_{t_j, t_{j + 1/2}}\otimes
	X_{t_{j + 1/2},t_{j + 1}}\\
	&\quad\,\, + \nabla f(Y^{h/2}_{t_j}) 
	\Big[Z(Y^{h/2}_{t_j})_{t_j, t_{j + 1/2}} 
	- Z(Y^{h/2}_{t_{j + 1/4}})_{t_j,t_{j + 1/2}}\Big]\otimes
	X_{t_{j + 1/2},t_{j + 1}}.
\end{align*}
The final summand above is $O(h^{(\gamma + 1)\alpha})$ 
by \eqref{eq:obs2}. Therefore, 
\begin{equation}\label{eq:A3_estm}
\begin{aligned}
A_3 & \overset{\eqref{eq:Z_sut}}{=}
	-\delta Z(Y^{h/2}_{t_j})_{t_j,t_{j + 1/2},t_{j + 1}}  + O(h^{\gamma \alpha}).
\end{aligned}
\end{equation}

We turn now to $I_2$. Using the observation 
in \eqref{eq:obs2} that shifting the argument in 
$Y^{h/2}$ incurs error of $O(h^{\gamma \alpha})$ for 
$Z(Y^{h/2})_{s,t}$, 
\begin{align}
I_2&\, = \delta Z(Y^{h/2}_{t_j})_{t_j,t_{j + 1/2},t_{j + 1}} 
	 + Z(Y^h_{t_j})_{t_j,t_{j + 1}} 
	-  Z(Y^{h/2}_{t_j})_{t_j,t_{j + 1}} + O(h^{\gamma \alpha})\notag\\
&  \overset{\eqref{eq:A3_estm}}{=} - A_3
	+ Z(Y^h_{t_j})_{t_j,t_{j + 1}} 
	-  Z(Y^{h/2}_{t_j})_{t_j,t_{j + 1}}
	+ O(h^{\gamma\alpha}).\label{eq:I2_estm}
\end{align}

 Let $\Delta Y_{t_j} := Y^h_{t_j} - Y^{h/2}_{t_j}$. 
Gathering $A_1$ from \eqref{eq:I1_dec}, 
$A_2$ from \eqref{eq:A2_estm}, 
$A_3$ from \eqref{eq:A3_estm}, 
and $I_2$ from \eqref{eq:I2_estm}, and 
inserting them into \eqref{eq:localerr}, 
we find that for a sufficiently big $C \ge 1$,
\begin{align}\label{eq:DeltaY_1}
\begin{aligned}
\big|\Delta Y_{t_j, t_{j + 1}}\big|
& \le   \big| \big(f(Y^h_{t_j}) - f(Y^{h/2}_{t_j})\big)
	 X_{t_j,t_{j + 1}} \\ &
	 \qquad\qquad+ Z(Y^h_{t_j})_{t_j,t_{j + 1}} 
	-  Z(Y^{h/2}_{t_j})_{t_j,t_{j + 1}} \big| + Ch^{\gamma \alpha}.
\end{aligned}
\end{align}
Estimating the absolute value on the right 
of \eqref{eq:DeltaY_1}, with \eqref{eq:obs2} and $f \in C^1$, we get
\begin{align*}
&\big| f(Y^h_{t_j}) - f(Y^{h/2}_{t_j})
	 X_{t_j,t_{j + 1}} + Z(Y^h_{t_j})_{t_j,t_{j + 1}} 
	-  Z(Y^{h/2}_{t_j})_{t_j,t_{j + 1}} \big|\\
& \qquad\qquad \qquad\qquad \qquad\qquad \qquad
	\le C\big|\Delta Y_{t_j}\big| h^\alpha  
	+ C \big|\Delta Y_{t_j}\big|^{\gamma - 2} h^{2\alpha}, 
\end{align*}
and \eqref{eq:DeltaY_1} implies 
\begin{align}\label{eq:iteration}
\begin{aligned}
\big|\Delta Y_{t_{j + 1}}\big| 
&\le \big|\Delta Y_{t_j, t_{j + 1}}\big| + \big|\Delta Y_{t_j}\big|\\
& \le    C\big|\Delta Y_{t_j}\big| h^\alpha  
	+ C \big|\Delta Y_{t_j}\big|^{\gamma - 2} h^{2\alpha} 
	+ Ch^{\gamma \alpha}  + \big|\Delta Y_{t_j}\big|.
\end{aligned}
\end{align}

Let us now argue by induction on the 
global error at time steps $t_k$ up to $k = j$. 
Suppose for a sufficiently small 
$h$ and  every $k \le j$, 
\begin{align*}
\abs{\Delta Y_{t_k}}	
	 \le  2C k h^{\gamma \alpha}.
\end{align*}	 
Conspiratorially, 
$\gamma^2 - 2\gamma + 2 > \gamma$ on 
$(\frac1\alpha, \infty) \subset (2,\infty)$, 
so for $\ep := \gamma^2 - 3 \gamma + 2 > 0$,
\begin{align*}
\big| \Delta Y_{t_{j + 1}}\big| 
	& \le 2C h^{(\gamma + 1)\alpha} 
	+ C(2C)^{\gamma - 2} h^{(\gamma + \ep)\alpha} 
	+ Ch^{\gamma \alpha} + 2 C j h^{\gamma \alpha}\\
	& = C(j + 1) h^{\gamma \alpha} 
	+ C \big( j + 2C h^\alpha + (2C)^{\gamma - 2} h^\ep\big) h^{\gamma \alpha}.
\end{align*}

Using the fact that $\Delta Y_0 = Y^h_0 - Y^{h/2}_0 = 0$, 
and supposing $h$ is small enough that 
$$
 2C h^\alpha + (2C)^{\gamma - 2} h^\ep \le 1
$$
completes the argument. We then have 
\begin{align*}
\max_{j}\abs{\Delta Y_{t_j}} 
\lesssim h^{\gamma \alpha - 1}.
\end{align*}
In the regime $\alpha \in (1/3, 1/2]$, we find 
$\gamma \alpha - 1 \le 3 \alpha - 1 \le \alpha$. 
Therefore, for $s \in (t_j, t_{j + 1})$ between 
time steps, using $\abs{Y^h_{s,t} }
	\lesssim \abs{t - s}^\alpha
\lesssim \abs{t - s}^{\gamma \alpha - 1}$ 
gives us \eqref{eq:unif_estm_r}.
\end{proof}

\subsection{Convergence rate in $C^\beta$}

Proposition \ref{thm:Yhh2_comparison} shows that  for $h_n = 2^{-n}$, 
the solutions $Y^{h_n}$ of the approximating systems 
form a Cauchy sequence in $L^\infty([0,T];\R^d)$, and 
also yields a convergence rate. Since
$Y^{h_n}$ are bounded in $C^\alpha([0,T];\R^d)$, the 
resulting limit is in $C^\alpha$. We briefly pursue convergence 
rates in $C^\beta$, $\beta < \alpha$.
\begin{cor}\label{thm:cauchy}
Let $Y^h$ be recursively defined via 
\eqref{eq:apsystem} and \eqref{eq:joinedup}. 
For any $\beta < \alpha$, set 
$\mu := (\alpha - \beta) \wedge (\gamma \alpha  -1)$.
The collection $\{Y^{2^{-n}}\}_{n = 1}^\infty$ 
form a Cauchy sequence in $C^\beta([0,T];\R^d)$ 
and for some constant indpendent of $n$, we have:
\begin{align*}
\abs{Y^{2^{-n}} - Y^{2^{-(n + 1)}}}_{C^\beta([0,T];\R^d)}  \le C 2^{-\mu n}.
\end{align*}
\end{cor}

\begin{proof}
Let us modify our notation slightly so that instead 
of writing $Y^{2^{-n}}$ for $Y^h$, {\em within this proof} 
we shall write $Y^n$.
For $ 0 \le s \le t \le T$, set
\begin{align*}
\mathcal{V}_{s,t}^n := \big| \big(Y^n_t - Y^{n + 1}_t\big)
- \big(Y^n_s - Y^{n + 1}_s\big)\big|.
\end{align*}

Let $t_j := j 2^{-n}$. 
Fix $n$ and assume that $s,t \in [0,T]$ such that for some $j$, 
and $[s,t]$ be a sub-interval of one of the intervals 
$[t_{j + k/4}, t_{j + {k + 1}/4}]$, $k \in \{0,1,2,3\}$.
Then $\mathcal{V}_{s,t}^n$ takes one of three forms. 
For $k = 0$, by Proposition \ref{thm:Yhh2_comparison}, 
\begin{align*}
\mathcal{V}_{s,t}^n& =\big| \big(f(Y^n_{t_j}) - f(Y^{n + 1}_{t_j}) \big)X_{s',t'} \big|
	\le C 2^{-(\gamma \alpha - 1)n} \abs{t - s}^{\alpha}, 
\end{align*}
where $s' := t_j + 2(s - t_j)$, $t' := t_j + 2(t - t_j)$, 
to account for running the scheme at ``twice the speed'' 
as in \eqref{eq:joinedup}. The important thing is that 
$|t - s| \sim |t' - s'|$ for any $n$. Let us additionally set 
$s'' = t_j + 2(s - t_{j + 1/4})$ and $t'' = t_j + 2(t - t_{j + 1/4})$. 
The case $k = 1$ (and similarly, $k = 2$) prevents 
$\mathcal{V}^n_{s,t}/|t - s|^\alpha$ 
from being bounded by $\sim 2^{- (\gamma \alpha - 1)n}$:
\begin{align*}
\mathcal{V}_{s,t}^n &= f(Y^n_{t_j}) X_{s',t'} 
	- Z(Y^{n + 1}_{t_{j + 1/4} })_{t_j,t''} +  Z(Y^{n + 1}_{t_{j + 1/4} })_{t_j,s''}  \\
& = f(Y^n_{t_j}) X_{s',t'} 
	- Z(Y^{n + 1}_{t_{j + 1/4} })_{t_j,t''} +  Z(Y^{n + 1}_{t_{j + 1/4} })_{t_j,s''} +    Z(Y^{n + 1}_{t_{j + 1/4} })_{s'',t''}  \\
&\quad\,\, + f(Y^{n + 1}_{t_{j + 1/4}}) \nabla f(Y^{n + 1}_{t_{j + 1/4}}) X_{t_j,s''}\otimes X_{s'',t''}   
	-    Z(Y^{n + 1}_{t_{j + 1/4} })_{s'',t''} \\
&\quad\,\,-  f(Y^{n + 1}_{t_{j + 1/4}}) \nabla f(Y^{n + 1}_{t_{j + 1/4}}) X_{t_j,s''}\otimes X_{s'',t''}     \\
&\quad\,\,	\le C|s - t|^{\gamma \alpha}
	+ \big|f(Y^n_{t_j}) X_{s',t'} \big| + 
	\big| Z(Y^{n + 1}_{t_{j + 1/4}})_{s'',t''}\big|\\
&\quad\,\,\qquad	  
	 + \big| \nabla f(Y^{n + 1}_{t_{j + 1/4}}) f(Y^{n + 1}_{t_{j + 1/4}})  X_{t_j,s''}\otimes X_{s'',t''}\big|
\le C  \abs{t - s}^{\alpha} \le C 2^{-(\alpha - \beta) n}\abs{t - s}^\beta.
\end{align*}
And finally, for $k = 3$,
\begin{align*}
\mathcal{V}_{s,t}^n & =  \big|Z(Y^{n + 1}_{t_{j + 1/2} })_{t_j,t''} 
	-  Z(Y^{n + 1}_{t_{j + 1/2} })_{t_j,s''}\\
&\qquad	-  Z(Y^{n + 1}_{t_{j + 3/4} })_{t_{j + 1/2},t''}  
	+ Z(Y^{n + 1}_{t_{j + 3/4} })_{t_{j + 1/2},s''}\big| \\
& \le C |t - s|^{\gamma \alpha} +\big| Z(Y^{n}_{t_{j}})_{s'',t''}
	- Z(Y^{n + 1}_{t_{j}})_{s'',t''}\big|\\
&\quad\,\, + \big|\nabla f(Y^{n}_{t_j}) f(Y^{n}_{t_j}) X_{s'',t''}X_{t_j,s''} 
	-\nabla f(Y^{n + 1}_{t_j}) f(Y^{n + 1}_{t_j}) X_{s'',t''}X_{t_{j + 1/2},s''} \big|\\
&	\le C |t - s|^{\gamma \alpha} 
	+ C \abs{t - s}^{2\alpha}
	+ C 2^{-\alpha n} \abs{t - s}^{\alpha}
\le C2^{-\alpha n}\abs{t - s}^{\alpha}.	
\end{align*}

Any interval $[s,t] \subset [0,T]$ can be decomposed 
into an almost disjoint union of intervals $[t_j, t_{j + 1}]$ 
and possibly two terminal intervals $[s, t_j]$ and $[t_j, t]$.
Therefore, 
$$
\sup_{[s,t] \subseteq [0,T]}  \frac{\mathcal{V}_{s,t}}{|t - s|^\beta} 
	\lesssim 2^{-(\alpha - \beta)n} \vee 2^{-(\gamma \alpha - 1)n}.
$$

\end{proof}

\subsection{Existence}\label{sec:sew}

It remains to show that the limit $Y$ is a solution 
in the sense of Davie (i.e., Definition \ref{def:solution}). 
This can be done abridging the argument of 
\cite[Claim in Lemma 2.4]{Dav2008}, included below 
for completeness. Following Davie, set
$$
J_{k m}^h := Y^h_{t_k,t_m}  - f(Y^h_{t_k}) X_{t_k,t_m} 
	- Z(Y^h_{t_k})_{t_k,t_m}.
$$
Obviously, $J^h_{kk} = 0$, and by  \eqref{eq:Z_xy}, 
$J^h_{k,k + 1} = O(h^{\gamma \alpha})$. 

For any $k \le \ell \le m$,
\begin{align*}
J^h_{k m} &= J^h_{k \ell} + J^h_{\ell m}  	
	+  f(Y^h)_{t_k,t_\ell}  X_{t_k,t_m}
	- \delta Z(Y^h_{t_k})_{t_k,t_\ell,t_m}\\
&\quad\,\, - \big[Z(Y^h_{t_k})_{t_\ell,t_m}
	- Z(Y^h_{t_\ell})_{t_\ell,t_m}\big].
\end{align*}
By manipulations similar to \eqref{eq:A3_estm}, 
and using \eqref{eq:obs2}, we find that
\begin{align*}
\abs{J^h_{k m}} &\le\abs{J^h_{k \ell}} + \abs{J^h_{\ell m}} 
	 + B \abs{t_m - t_k}^{\gamma \alpha},
\end{align*}
for some absolute constant $B > 0$, dependent 
only on $f$ and $Z$. Then by induction (on $m- k$), 
$$
\abs{J^h_{k m}} \le L \big(\abs{t_\ell - t_k}^{\gamma \alpha} 
	+ \abs{t_m - t_{\ell + 1}}^{\gamma \alpha}\big) 
	+ B \abs{t_m - t_k}^{\gamma \alpha}.
$$
With a judicious choice of $\ell$, we can enforce 
\begin{align*}
\abs{t_\ell - t_k} ,\abs{t_m - t_{\ell + 1}} \le \frac12 \abs{t_m - t_k}.
\end{align*}
Recall that $\gamma \alpha > 1$. 
By choosing $L$ sufficiently  large such that 
\begin{align*}
L 2^{1-\gamma \alpha} + B \le L,
\end{align*}
we find $\abs{J^h_{k m}}  \le L \abs{t_m - t_k}^{\gamma \alpha}$. 
In particular, $L$ is independent of $h$, $k$ and $m$. 
This shows that our limits are in fact solutions in 
the sense of Definition \ref{def:solution}.

\section{A further remark on convergence}\label{sec:unique}

In Section \ref{sec:wellposed} we inspected the convergence of a 
very particular subsequence of $\{Y^h\}_{h > 0}$.
One might be tempted to think that arguments 
similar to those in Proposition \ref{thm:Yhh2_comparison} 
can yield 
%
%
%
%
the more general result:
\begin{prop}\label{thm:eq18}
Let $q \in \mathbb{Q} \cap [1,2)$. Set 
$\Delta Y  := Y^h - Y^{h/q}$, with $Y^h$ and $Y^{h/q}$ 
recursively defined via 
\eqref{eq:apsystem} and \eqref{eq:joinedup}.
For some constant $C_q$ dependent on the 
numerator of $q$ in lowest terms,
\begin{align}\label{eq:rationalq_improve}
\big\|\Delta Y \big\|_{L^\infty([0,T])}
	\le C_q h^{\gamma\alpha - 1} 
\end{align}
\end{prop} 
Establishing the proposition turns out to require 
more than a slight modification on the 
proof of Proposition \ref{thm:Yhh2_comparison} 
and we dedicate the remainder of this note to 
obtaining this result, which may be of independent 
interest to the analysis of splitting schemes.
Before doing so, we mention that 
there is no real loss of generality in assuming $q \in [1,2)$ 
as Proposition \ref{thm:Yhh2_comparison} itself allows us to 
move up and down the dyadic scale. Moreover, 
Proposition \ref{thm:eq18} has a slightly less restrictive 
formulation as it implies that for any fixed $n$, 
\begin{align*}
\|Y^h - Y^{h/q^n}\|_{L^\infty} 
\le \sum_{k < n}\| Y^{h/q^{k + 1}} -  Y^{h/q^k}\|_{L^\infty} 
\le C_q  \frac{q^{\gamma \alpha - 1}}{q^{\gamma \alpha - 1} - 1} 
	h^{\gamma \alpha - 1}.
\end{align*}

\subsection{Proof of  
	\eqref{eq:rationalq_improve}}
Let $1 \le q = k/m \le 2$ in lowest terms.
Let $t_j := j h$ and let $s_j := j h/q$.  Designate by  
$\mathcal{T} := \{\tau_\ell\}_{\ell = 1}^{k + m - 1}$
 the collection, indexed in an ascending order, 
containing all times $nmh < s_j, t_j \le  (n + 1) mh$. 
As $k/m$ expresses $q$ in lowest terms, the time-steps 
\begin{align}\label{eq:T_sT_t_defin}
\mathcal{T}_s := \{s_{nk + j}\}_{j = 1}^{k - 1} \quad\text{ and }\quad
\mathcal{T} _t := \{t_{nm + j}\}_{j = 1}^{m}
\end{align}
are disjoint as we have elected to remove the 
end-point from $\mathcal{T}_s$. Together, 
$\mathcal{T}_s \cup \mathcal{T}_t$ is $\mathcal{T}$. 
For convenience, we additionally take $\tau_0 := t_{nm}$.

A final notational construct before heading into 
our proof, for a point 
$\tau_{\ell - 1} \in \mathcal{T}$, 
let $\hat{\tau}_\ell\le \tau_{\ell - 1} + h/2$ 
if $\tau_{\ell - 1} \in \mathcal{T}_t$ 
(or $\hat{\tau}_\ell\le \tau_{\ell - 1} + h/(2q)$ 
if $\tau_{\ell - 1} \in \mathcal{T}_s$) be the point for which 
\begin{align}\label{eq:shiftedtau}
\tau_\ell = \tau_{\ell - 1} + 2(\hat{\tau}_\ell - \tau_{\ell - 1}).
\end{align}
\begin{figure}
\includegraphics[width = 2.3in]{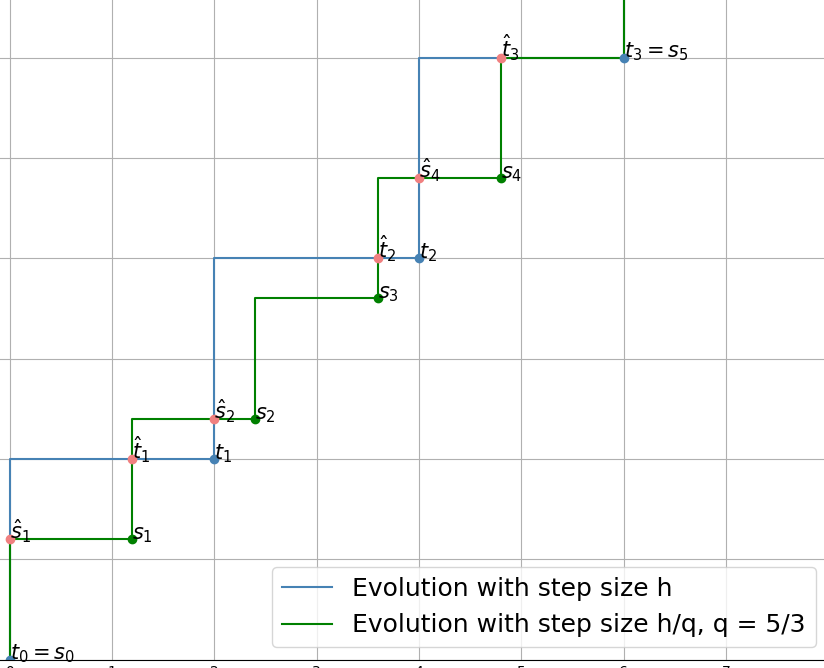}
\caption{Schematic plot of the splitting scheme, 
with vertical traversal signifying evolution by the 
first equation, and horizontal traversal signifying 
evolution under the second equation.}
\label{fig:splitup1}
\end{figure}
This accounts for running the scheme 
at ``twice the speed'', sped up linearly 
(see \eqref{eq:joinedup}, and pictorally, 
Figure \ref{fig:splitup1}), so that for $\tau_{\ell - 1} \in \mathcal{T}_t$, 
$$
Y^a_{\hat{\tau}_\ell} = Y^h_{\tau_{\ell  - 1}} + f(Y^h_{\tau_{\ell  - 1}}) 
	X_{\tau_{\ell - 1},\tau_{\ell}},
$$
with $h/q$ replacing $h$ if $\tau_{\ell - 1} \in \mathcal{T}_s$.
By linearity of the speed-up, we clearly also have 
(if $\tau_\ell \in \mathcal{T}_t$) 
\begin{align*}
Y^h_{\hat{\tau}_\ell} = Y^h_{\tau_{\ell} - h/2}
	+  Z(Y^h_{\tau_{\ell} - h/2})_{\tau_{\ell} - h,\tau_{\ell - 1}},
\end{align*}
again with $h/q$ replacing $h$ if $\tau_\ell \in \mathcal{T}_s$. 

Our key estimate is analogous to \eqref{eq:iteration}:
\begin{lem}
Let $\Lambda_1$ denote the collection of points 
$\tau_\ell \in \mathcal{T}_s$ such that 
$\tau_{\ell - 1} \in \mathcal{T}_s$ 
(see also Remark \ref{rem:Lambda0} below).  
With the notation established as immediately 
foregoing, we have the estimate:
\begin{align}\label{eq:keyestm_qinQ}
\Delta Y_{t_{nm},t_{(n + 1)m}} 
 = O(h^{\gamma\alpha})
 	+ O(|\Delta Y_{t_{nm}}|^{\gamma - 2} h^{2\alpha})
 	+ \sum_{\tau_\ell \in \mathcal{T} \backslash \Lambda_1}
 	O(|\Delta Y_{\hat{\tau}_{\ell}}| h^\alpha),
\end{align}
where the implicit constants depend on 
$k + m$.
\end{lem}

\begin{rem}\label{rem:Lambda0}
Since $h/q \le h$, if $\tau_\ell \in \mathcal{T}_t$, then 
$\tau_{\ell - 1},\tau_{\ell + 1} \in \mathcal{T}_s$. That is, there is at most 
one point $t_{nm + j'} = \tau_\ell \in \mathcal{T}$ between 
any pair $s_{nk + j}$ and $s_{nk + j + 1}$. 
We shall also find it convenient to define 
\begin{align}\label{eq:Lambda0_defin}
\Lambda_0 = \{\tau_\ell \in \mathcal{T}_s : \tau_{\ell + 1} \in \mathcal{T}_s\}.
\end{align}
Pictorially, in Figure \ref{fig:splitup1}, 
$s_2$ is the only member of $\Lambda_0$ 
over the time step $t_0$ to $t_5$. 
We also have then that $\tau_\ell \in \Lambda_0$ 
if and only if $\tau_{\ell + 1} \in \Lambda_1$. 
\end{rem}

\begin{proof}

We can write the increment for $Y^{h/q}$ over 
$k$ time steps as:
\begin{align}\label{eq:Yq_incr_qinQ}
\begin{aligned}
Y^{h/q}_{s_{nk},s_{(n + 1)k}}
&=\sum_{j = 0}^{k - 1} \Big(f(Y^{h/q}_{s_{nk + j}}) 
	X_{s_{nk + j},s_{nk + j + 1}}\\
&\qquad\qquad\,\, 
+ Z(Y^{h/q}_{s_{nk + j + 1/2}})_{s_{nk + j},s_{nk + j + 1}}\Big) 
 = A^q_1 + A^q_2 + A^q_3, 
\end{aligned}
\end{align}
where
\begin{align*}
A^q_1 &:=  \sum_{j = 0}^{k - 1} f(Y^{h/q}_{s_{nk + j}}) 
	\sum_{s_{nk + j} \le \tau_\ell < s_{nk + j + 1}} 
	X_{\tau_\ell,\tau_{\ell + 1}}, \\
A^q_2 &:= \sum_{j = 0}^{k - 1}
	 \Big( Z(Y^{h/q}_{s_{nk + j + 1/2}})_{s_{nk + j},s_{nk + j + 1}}
	  - \sum_{s_{nk + j} \le \tau_\ell < s_{nk + j + 1}} 
	 Z(Y^{h/q}_{s_{nk + j + 1/2}})_{\tau_{\ell},\tau_{\ell + 1}} \Big),\\
A^q_3 &:=   \sum_{j = 0}^{k - 1}
 \sum_{s_{nk + j} \le \tau_\ell < s_{nk + j + 1}} 
	 Z(Y^{h/q}_{s_{nk + j + 1/2}})_{\tau_{\ell},\tau_{\ell + 1}}.
\end{align*}
We now make two observations. First, by \eqref{eq:Z_xy}, 
we can always write $A^q_3$ as
\begin{align}\label{eq:Aq3}
A^q_3 =  \sum_{\ell = 0}^{km - 1}  
	Z(Y^{h/q}_{s_{nk}})_{\tau_{\ell},\tau_{\ell + 1}} + O(h^{\gamma \alpha}).
\end{align}
We can subtract these from similar terms in 
the increment $Y^{h}(t_{(n + 1)m}) - Y^{h}(t_{nm})$ later 
in \eqref{eq:A3_qinQ}. 
Next, if $s_{nk + j} \in \Lambda_0$, 
$$
 Z(Y^{h/q}_{s_{nk + j + 1/2}})_{s_{nk + j},s_{nk + j + 1}} \\
 = \sum_{s_{nk + j} \le \tau_\ell < s_{nk + j + 1}} 
 	 Z(Y^{h/q}_{s_{nk + j + 1/2}})_{\tau_{\ell},\tau_{\ell + 1}},
$$
as there is only one $\tau_\ell = s_{nk + j}$ in the 
interval $[s_{nk + j}, s_{nk + j + 1})$, 
and the $j$th summand of $A^q_2$ is nought. 
When $s_{nk + j} \not\in \Lambda_0$, the $j$th summand of $A^q_2$ is 
\begin{align}\label{eq:Aq2_ZZZ}
\begin{aligned}
&\delta Z(Y^{h/q}_{s_{nk + j + 1/2}})_{s_{nk + j}, \tau_\ell,s_{nk + j + 1}}\\
&\qquad\overset{\eqref{eq:Z_sut}}{=} 
f(Y^{h/q}_{s_{nk + j}}) \nabla f(Y^{h/q}_{s_{nk + j}}) 
	X_{\tau_\ell,s_{nk + j + 1}}
	\otimes X_{s_{nk + j},\tau_\ell} + O(h^{3\alpha}).
\end{aligned}
\end{align}
On $\mathcal{T}_s \backslash \Lambda_0$, 
we will need to subtract summands in $A^q_1$ from 
corresponding terms in the decomposition for 
$Y^h_{t_{nm},t_{(n + 1)m}}$  in order to 
handle \eqref{eq:Aq2_ZZZ}.\vspace{.1cm}

Turning now to just that decomposition, 
as in \eqref{eq:Yq_incr_qinQ}, we write the 
increment $Y^h_{t_{nm},t_{(n + 1)m}}$ 
as  $A^h_1 + A^h_2 +A^h_3$, with
\begin{align*}
A^h_1 & := \sum_{j = 0}^{m - 1} f(Y^h_{t_{nm + j}}) 
	\sum_{t_{nm + j} \le \tau_\ell < t_{nm + j + 1}} 
	X_{\tau_\ell,\tau_{\ell + 1}}, \\
A^h_2 &:= \sum_{j = 0}^{m - 1} 
	\Big( Z(Y^h_{t_{nm + j + 1/2}})_{t_{nm + j},t_{nm + j + 1}} 
	 - \sum_{t_{nm + j} \le \tau_\ell < t_{nm + j + 1}} 
	 Z(Y^h_{t_{nm + j + 1/2}})_{\tau_{\ell},\tau_{\ell + 1}} \Big),\\
A^h_3&:= \sum_{\ell = 0}^{km - 1}  
	Z(Y^h_{t_{nm}})_{\tau_{\ell},\tau_{\ell + 1}} + O(h^{\gamma \alpha}).
\end{align*}

Using $t_{nm} = nmh = nkh/q = s_{nk}$ and \eqref{eq:Z_xy},
subtracting $A^q_3$ in \eqref{eq:Aq3} from $A^h_3$, we find
\begin{align}\label{eq:A3_qinQ}
|A^h_3 - A^q_3| 
	\le C |\Delta Y_{t_{nm}}|^{\gamma - 2} h^{2\alpha} + Ch^{\gamma \alpha},
\end{align}

Recall the definitions of $\mathcal{T}_s$ and $\mathcal{T}_t$ 
in \eqref{eq:T_sT_t_defin}, and that of $\Lambda_0$ 
in \eqref{eq:Lambda0_defin}.  For $A^h_1$, we can 
decompose the difference as follows:
\begin{align}\label{eq:Ah1}
\begin{aligned}
A^h_1 - A^q  _1
& = \sum_{\tau_\ell \in \mathcal{T}_t} f(Y^h_{\tau_\ell}) 
	\big[X_{\tau_\ell,\tau_{\ell + 1}}
	+ X_{\tau_{\ell + 1},\tau_{\ell + 2}}
	+ X_{\tau_{\ell + 2},\tau_{\ell + 3}}
	\mathds{1}_{\{\tau_{\ell+ 1} \in \Lambda_0\}}
	\big]\\
&\quad\,\, -  \sum_{\tau_\ell \in \mathcal{T}_s} f(Y^{h/q}_{\tau_\ell}) 
	\big[X_{\tau_\ell,\tau_{\ell + 1}}
	 + X_{\tau_{\ell + 1},\tau_{\ell + 2}} 
	\mathds{1}_{\{\tau_{\ell} \not\in \Lambda_0\}}\big]\\
& = H_1 + H_2 + H_3 + H_4, 
\end{aligned}
\end{align}
where
\begin{align*}
H_1 &:= \sum_{\substack{\tau_\ell \in \mathcal{T}_s \\ 
		\tau_{\ell  + 1} \in \mathcal{T}_t}} 
	\big(f(Y^h_{\tau_{\ell + 1}}) - f(Y^{h/q}_{\tau_\ell})\big) 
	X_{\tau_{\ell + 1},\tau_{\ell +2}},\\
H_2 & := \sum_{\substack{\tau_\ell \in \mathcal{T}_t \\ 
		\tau_{\ell  + 1} \in \mathcal{T}_s \backslash\Lambda_0}} 
	\big(f(Y^h_{\tau_\ell}) - f(Y^{h/q}_{\tau_{\ell + 1}})\big) 
	X_{\tau_{\ell + 1},\tau_{\ell + 2}},\\
H_3 &:= \sum_{\substack{\tau_\ell \in \mathcal{T}_t \\ 
		\tau_{\ell  + 1} \in \Lambda_0} }
	\big(f(Y^h_{\tau_\ell}) - f(Y^{h/q}_{\tau_{\ell + 1}})\big) 
	X_{\tau_{\ell + 1},\tau_{\ell +3}},\\
H_4 & :=  \sum_{\substack{\tau_\ell \in \mathcal{T}_t \\ 
		\tau_{\ell  + 1} \in \Lambda_0} }
	f(Y^{h/q})_{\tau_{\ell + 2},\tau_{\ell + 1}} 
	X_{\tau_{\ell + 2},\tau_{\ell +3}}.
\end{align*}
Next recall the shifted points 
$\hat{\tau}_\ell$ defined in \eqref{eq:shiftedtau}. 
The aim of all subsequent manipulations is to 
consider the difference $\Delta Y$ at these shifted 
points where $Y^h$ and $Y^{h/q}$ have each been 
propagated by the same amount of time in the 
first equation (and in the second equation) 
of the splitting scheme.
We re-arrange the terms in $H_1$:
\begin{align*}
 H_1 &= \sum_{\tau_{\ell  + 1} \in \mathcal{T}_t} 
	f(Y^h)_{\hat{\tau}_{\ell + 1},\tau_{\ell + 1}} 
	X_{\tau_{\ell + 1},\tau_{\ell + 2}}
	+ \sum_{\tau_{\ell  + 1} \in \mathcal{T}_t} 
	f(Y^{h/q})_{\tau_\ell,\hat{\tau}_{\ell + 1}}
	X_{\tau_{\ell + 1},\tau_{\ell + 2}}\\
&\quad\,\, + \sum_{\tau_{\ell  + 1} \in \mathcal{T}_t} 
	\big(f(Y^h_{\hat{\tau}_{\ell + 1}}) - f(Y^{h/q}_{\hat{\tau}_{\ell + 1}})\big) 
	X_{\tau_{\ell + 1},\tau_{\ell + 2}}\\
& = O(h^{\gamma\alpha}) + \sum_{\tau_{\ell  + 1} \in \mathcal{T}_t} 
	O(|\Delta Y_{\hat{\tau}_{\ell + 1}}| h^\alpha) \\
	&\qquad\qquad\,\, \,+ \sum_{\tau_{\ell  + 1} \in \mathcal{T}_t} 
	\nabla f(Y^{h/q}_{\tau_\ell})  f(Y^{h/q}_{\tau_\ell}) 
	X_{\tau_\ell,\tau_{\ell + 1}}\otimes
	X_{\tau_{\ell + 1},\tau_{\ell + 2}}.
\end{align*}
And specifically, in the estimate for $H_1$ 
above, we used the implication $\tau_{\ell + 1} 
\in \mathcal{T}_t \implies \tau_\ell \in \mathcal{T}_s$ 
(see Remark \ref{rem:Lambda0}), and 
\begin{align*}
\big|f(Y^h)_{\hat{\tau}_{\ell + 1},\tau_{\ell + 1}} \big|
& \le C \big| Y^h_{\hat{\tau}_{\ell + 1}, \tau_{\ell + 1}}\big| \\
&= C \big| Z(Y^{h}_{\tau_{\ell + 1} - h/2})_{\tau_{\ell + 1} - h, \tau_{\ell + 1}} 
- Z(Y^{h}_{\tau_{\ell + 1} - h/2})_{ \tau_{\ell + 1} - h, \tau_{\ell - 1}} \big|\\
&\overset{\eqref{eq:Z_st}}{=} O(h^{2\alpha}).
\end{align*}
By comparison with the terms of 
$A^q_2$ in \eqref{eq:Aq2_ZZZ}, we find
\begin{align}\label{eq:H1Aq2}
H_1 - A^q_2  =  O(h^{\gamma\alpha}) 
	+ \sum_{\substack{\tau_\ell \in \mathcal{T}_s \\ 
		\tau_{\ell  + 1} \in \mathcal{T}_t}} 
	O(|\Delta Y_{\hat{\tau}_{\ell + 1}}| h^\alpha).
\end{align}

The remaining $H_i$, $i = 2,3,4$,  are used to 
tame  $A^h_2$. Similar calculations yield:
\begin{align*}
H_2 & = O(h^{\gamma\alpha}) 
	+ \sum_{\substack{\tau_\ell \in \mathcal{T}_t \\ 
		\tau_{\ell  + 1} \in \mathcal{T}_s\backslash \Lambda_0}} 
	O(|\Delta Y_{\hat{\tau}_{\ell + 1}}| h^\alpha) \\
	&\qquad\qquad\,\,\, - \sum_{\substack{\tau_\ell \in \mathcal{T}_t \\ 
		\tau_{\ell  + 1} \in \mathcal{T}_s\backslash \Lambda_0}} 
	\nabla f(Y^{h}_{\tau_\ell})  f(Y^{h}_{\tau_\ell}) 
	X_{\tau_\ell,\tau_{\ell + 1}}\otimes
	X_{\tau_{\ell + 1},\tau_{\ell + 2}},\\
H_3 & = O(h^{\gamma\alpha}) 
	+ \sum_{	\tau_{\ell  + 1} \in \Lambda_0} 
	O(|\Delta Y_{\hat{\tau}_{\ell + 1}}| h^\alpha) \\
	&\qquad\qquad\,\,\, - \sum_{\tau_{\ell  + 1} \in \Lambda_0} 
	\nabla f(Y^{h}_{\tau_\ell})  f(Y^{h}_{\tau_\ell}) 
	X_{\tau_\ell,\tau_{\ell + 1}}\otimes
	X_{\tau_{\ell + 1},\tau_{\ell + 2}}, \\
H_4 & = O(h^{\gamma\alpha}) - \sum_{\tau_{\ell  + 1} \in \Lambda_0} 
	\nabla f(Y^{h}_{\tau_\ell})  f(Y^{h}_{\tau_\ell}) 
	X_{\tau_{\ell + 1}, \tau_{\ell + 2}}	\otimes
	X_{\tau_{\ell + 2},\tau_{\ell + 3}}.
\end{align*}

We turn finally to $A^h_2$ itself. As in $A^q_2$ above, 
the summands of $A^h_2$ take one of two forms:
\begin{align}\label{eq:Ah2}
A^h_2 = H_5 + H_6,
\end{align}
where
\begin{align*}
H_5 &:= \sum_{\substack{\tau_\ell \in \mathcal{T}_t\\
	\tau_{\ell + 1} \in \mathcal{T}_s \backslash \Lambda_0}} 
	\delta Z(Y^h)_{\tau_{\ell}, \tau_{\ell + 1},\tau_{\ell} + h},\\
H_6 &:= \sum_{\tau_{\ell + 1} \in \Lambda_0} 
	\big(Z(Y^h)_{\tau_{\ell}, \tau_{\ell} + h} 
	- Z(Y^h)_{\tau_{\ell}, \tau_{\ell + 1}} \\
&\qquad\qquad\qquad - Z(Y^h)_{\tau_{\ell + 1}, \tau_{\ell + 2}}
	- Z(Y^h)_{\tau_{\ell + 2}, \tau_{\ell} + h}\big)\\
& = \sum_{\tau_{\ell + 1} \in \Lambda_0} 
	\big(\delta Z(Y^h)_{\tau_{\ell}, \tau_{\ell + 1}, \tau_{\ell} + h}
	+ \delta Z(Y^h)_{\tau_{\ell + 1}, \tau_{\ell + 2}, \tau_{\ell} + h}\big)\\
\end{align*}
We deliberately omitted the argument for $Y^h$ 
above as, again, shifting that argument incurs only an error 
of $O(h^{\gamma \alpha})$, and we do not need  
detailed bookkeeping here. Up to this error, 
we get 
\begin{align*}
H_5 + H_2 &= O(h^{\gamma\alpha}) 
	+ \sum_{\substack{\tau_\ell \in \mathcal{T}_t \\ 
		\tau_{\ell  + 1} \in \mathcal{T}_s\backslash \Lambda_0}} 
O(|\Delta Y_{\hat{\tau}_{\ell + 1}}| h^\alpha), \\
H_6 + H_3 + H_4 &=  O(h^{\gamma\alpha}) 
	+ \sum_{	\tau_{\ell  + 1} \in \Lambda_0} 
	O(|\Delta Y_{\hat{\tau}_{\ell + 1}}| h^\alpha).
\end{align*}
Putting these together with \eqref{eq:A3_qinQ}, 
\eqref{eq:Ah1}, \eqref{eq:H1Aq2}, and \eqref{eq:Ah2}, 
we arrive at \eqref{eq:keyestm_qinQ}.
\end{proof}

In order to close estimates, we need to 
be able to control $\Delta Y$ at shifted 
times $\hat{\tau}_\ell$. 
\begin{lem}\label{lem:majorising}
Let $\mathcal{T} = \mathcal{T}_t \cup \mathcal{T}_s$ 
be as defined in \eqref{eq:T_sT_t_defin}. 
Let $\Lambda_1$ be constituted of points 
$\tau_\ell \in \mathcal{T}_s$ for which 
$\tau_{\ell - 1} \in \mathcal{T}_s$. For 
$\tau_\ell \in \mathcal{T}$, $\tau_{\ell + 1}
	\not \in \Lambda_1$, we have the estimate
\begin{align*}
|\Delta Y_{\hat{\tau}_{\ell + 1}}| 
\lesssim |\Delta Y_{\hat{\tau}_{\ell - 1}}| 
\mathds{1}_{\tau_{\ell} \in \Lambda_1} 
 +  |\Delta Y_{\hat{\tau}_\ell}| 
\mathds{1}_{\tau_{\ell} \not\in \Lambda_1}  + h^{2\alpha},
\end{align*}
and in particular, the implicit constant is independent of $\ell$. 
\end{lem}

\begin{proof}
Let us compute this bound for $\tau_{\ell} \in \mathcal{T}_t $ 
and $\tau_{\ell + 1} \in \mathcal{T}_s$. The remaining two cases: (i) 
$\tau_{\ell } \in \Lambda_1$, and (ii) $\tau_{\ell} 
	\in\mathcal{T}_s \backslash (\Lambda_0 \cup \Lambda_1)$, 
depend on similar calculations and 
we will not belabour the point.  

In the present case we can use the simplification 
$\tau_{\ell + 1} - h/q = \tau_{\ell - 1} \in \mathcal{T}_s$. 
Using other properties of $\hat{\tau}$ summarised 
in \eqref{eq:shiftedtau}, we get
$
\Delta Y_{\hat{\tau}_\ell, \hat{\tau}_{\ell + 1}} = G_1 + G_2, 
$
where
\begin{align*}
G_1 & := \big(f(Y^h_{\tau_\ell}) - f(Y^{h/q}_{\tau_{\ell - 1}})\big) 
	X_{\tau_\ell,\tau_{\ell + 1}},\\
G_2 &:=  Z(Y^h_{\tau_{\ell} - h/2})_{\tau_{\ell} - h, \tau_{\ell}} 
	- Z(Y^h_{\tau_{\ell} -h/2})_{\tau_{\ell} - h, \tau_{\ell - 1}}\\
& \quad\,\, 
	- Z(Y^{h/q}_{\tau_{\ell + 1} - h/(2q)})_{\tau_{\ell - 1}, \tau_{\ell + 1}} 
	+ Z(Y^{h/q}_{\tau_{\ell + 1} - h/(2q)} )_{\tau_{\ell}, \tau_{\ell + 1}}. 
\end{align*}
We consider only $G_1$ below, since the remainder 
are $O(h^{2\alpha})$:
\begin{align*}
G_1 &= f(Y^h)_{\hat{\tau}_\ell,\tau_\ell}X_{\tau_\ell,\tau_{\ell + 1}}
	+\big( f(Y^h_{\hat{\tau}_\ell}) -  f(Y^{h/q}_{\hat{\tau}_\ell})\big)
	X_{\tau_\ell,\tau_{\ell + 1}} 
	+ f(Y^{h/q})_{\tau_{\ell - 1},\hat{\tau}_\ell}
	X_{\tau_\ell,\tau_{\ell + 1}}\\
& = O(|\Delta Y_{\hat{\tau}_\ell}| h^\alpha) + O(h^{2\alpha}),
\end{align*}
where the $h^{2\alpha}$ asymptotic controls 
the first and last summands in the line preceding. This establishes the lemma.
\end{proof}

Wrapping up briefly:
\begin{proof}[Proof of Proposition \ref{thm:eq18}]
Applying Lemma \ref{lem:majorising} to 
\eqref{eq:keyestm_qinQ}, 
along with the straightforward end-point estimate 
$\big|\Delta Y_{\hat{\tau}_1}\big| \lesssim \big|\Delta Y_{t_{nm}}\big|$, 
for some constant $C_{k,m} \ge 1$, 
we arrive at the following estimate: 
\begin{align*}
\big|\Delta Y_{t_{nm},t_{(n + 1)m}} \big|
 \le C_{k,m} h^{\gamma\alpha}
 	+ C_{k,m} \big|\Delta Y_{t_{nm}}\big|^{\gamma - 2} h^{2\alpha}
 	+ C_{k,m} \big|\Delta Y_{t_{nm}}\big| h^\alpha.
\end{align*}
This give us \eqref{eq:rationalq_improve} via the 
induction argument presented in the proof of Lemma 
\ref{thm:Yhh2_comparison} proceeding from 
\eqref{eq:iteration}.
\end{proof}

\bibliographystyle{plain}

\end{document}